\documentclass[11pt]{article}
\usepackage{amsmath,amssymb,amsthm}
\usepackage{fullpage}
\usepackage{dsfont}
\usepackage[round]{natbib}
\usepackage{hyperref}

\title{On Translation Invariant Kernels and Screw Functions}
\author{Purushottam Kar and Harish Karnick\\Department of Computer Science and Engineering\\Indian Institute of Technology Kanpur\\{\tt \{purushot,hk\}@cse.iitk.ac.in}}
\date{\today}

\newcommand{\R}{{\mathbb R}}
\newcommand{\N}{{\mathbb N}}
\newcommand{\Z}{{\mathbb Z}}
\renewcommand{\H}{\mathfrak H}

\newcommand{\br}[1]{\left({#1}\right)}
\newcommand{\bs}[1]{\left[{#1}\right]}
\newcommand{\abs}[1]{\left| {#1} \right|}
\newcommand{\norm}[1]{\left\| {#1} \right\|}

\newcommand{\ip}[2]{\left\langle{#1},{#2}\right\rangle}

\renewcommand{\vec}[1]{{\mathbf{#1}}}
\newcommand{\vecx}{\vec{x}}
\newcommand{\vecy}{\vec{y}}
\newcommand{\vect}{\vec{t}}
\newcommand{\vecxy}{\vecx,\vecy}

\newcommand{\veczero}{\vec{0}}
\newcommand{\vecone}{\vec{1}}

\renewcommand{\th}{^{\text{th}}}

\newcommand{\ind}{\mathds{1}}

\newtheorem{lem}{Lemma}
\newtheorem{thm}[lem]{Theorem}
\newtheorem{cor}[lem]{Corollary}

\begin{document}


\maketitle

\setcounter{footnote}{1}

\begin{abstract}
We explore the connection between Hilbertian metrics and positive definite kernels on the real line. In particular, we look at a well-known characterization of translation invariant Hilbertian metrics on the real line by \citet{helices}. Using this result we are able to give an alternate proof of Bochner's theorem for translation invariant positive definite kernels on the real line \citep{faog}.
\end{abstract}

\section{Introduction}
We start off with a few definitions and set the notation. In the following discussion, $\H$ shall denote the real Hilbert space, small case letters $x,y,\ldots$ shall denote real numbers, boldface small letters $\vecxy,\ldots$ shall denote objects in arbitrary domains. A metric $D$ defined on some domain $\Omega$ will be said to be \emph{Hilbertian} if there exists a map $\xi : \Omega \rightarrow \H$ such that for all $\vecxy \in \Omega$, $D(\vecxy) = \norm{\xi(\vecx) - \xi(\vecy)}_\H$. A symmetric real valued kernel $K : \Omega \times \Omega \rightarrow \R$ will be said to be \emph{positive definite} if for any $n \in \N$, any $\vecx_1,\ldots,\vecx_n \in \Omega$, if we consider the matrix $G_K = \bs{K(\vecx_i,\vecx_j)}$, then for any $c \in \R^n$, we have $c^\top G_K c \geq 0$. A symmetric real valued kernel $N : \Omega \times \Omega \rightarrow \R$ will be said to be \emph{negative definite} if whenever $c^\top \vecone = 0$, we have $c^\top G_N c \leq 0$ (here $\vecone \in \R^n$ is the vector of all ones). It is easy to verify that all squared Hilbertian metrics are negative definite.

There exists a close connection between negative definite kernels and positive definite kernels as given below :
\begin{thm}[\citep{bcr}, Chapter 3, Lemma 2.1]
\label{psd-nd}
For any given kernel $N$ over $\Omega$, and some fixed $\veczero \in \Omega$, define a new kernel $K$ as $K(\vecxy) = \frac{1}{2}\br{N(\vecx,\veczero) + N(\vecy,\veczero) - N(\vecxy) - N(\veczero,\veczero)}$. Alternatively if we have $N(\veczero,\veczero) \geq 0$ and define $K(\vecxy) = \frac{1}{2}\br{N(\vecx,\veczero) + N(\vecy,\veczero) - N(\vecxy)}$. Then $N$ is negative definite iff $K$ is positive definite.
\end{thm}

Using the above expression for $K(\vecxy)$, it is simple to arrive at the following relation
\[
-\frac{1}{2}\br{N(\vecx,\vecx) + N(\vecy,\vecy) - 2N(\vecxy)} = K(\vecx,\vecx) + K(\vecy,\vecy) - 2K(\vecxy)
\]
In case $N$ is a squared Hilbertian metric satisfying identity of the indiscernibles (i.e. $N(\vecx,\vecx) = 0$ for all $x \in \Omega$), we can obtain the following relation $N(\vecxy) = K(\vecx,\vecx) + K(\vecy,\vecy) - 2K(\vecxy)$. This allows us to arrive at the following simple but useful corollary:

\begin{cor}
\label{psd-hilb}
Given a distance function $D$ over $\Omega$ satisfying identity of the indiscernibles and some fixed $\veczero \in \Omega$, define a kernel $K$ as $K(\vecxy) = \frac{1}{2}\br{D^2(\vecx,\veczero)  + D^2(\vecy,\veczero) - D^2(\vecxy)}$. Then $D$ is Hilbertian iff $K$ is positive definite.
\end{cor}
\begin{proof}
To see that positive definiteness of $K$ is necessary, simply invoke Theorem~\ref{psd-nd} along with the fact that since $D$ is a metric, $D^2(\vecx,\vecx) \geq 0$ for all $x \in \Omega$ and that all squared Hilbertian metrics are negative definite. To see that is sufficient, simply invoke Mercer's theorem (\citep{mercer} or equivalently the discussion in \citep{bcr}, Chapter 3, Section 3) to confirm the existence of a mapping $\varphi : \Omega \rightarrow \H$ such that $K(\vecxy) = \ip{\varphi(\vecx)}{\varphi(\vecy)}_\H$. Thus we have $\norm{\varphi(\vecx) - \varphi(\vecy)}_\H^2 = K(\vecx,\vecx) + K(\vecy,\vecy) - 2K(\vecxy)$. Next, using the definition of $K$, we conclude that $K(\vecx,\vecx) + K(\vecy,\vecy) - 2K(\vecxy) = D^2(\vecxy) - \frac{1}{2}\br{D^2(\vecx,\vecx) + D^2(\vecy,\vecy)} = D^2(\vecxy)$ since $D$ satisfies identity of the indiscernibles. This gives us $\norm{\varphi(\vecx) - \varphi(\vecy)}_\H^2 = D^2(\vecxy)$ thus proving that $D$ is in fact a Hilbertian metric.
\end{proof}

The preceding corollary and its proof give us the well known correspondence between the set of positive definite kernels and Hilbertian metrics \citep{kernel-trick-distances}. Given any positive definite kernel $K$ it is possible to construct a Hilbertian metric $D_K$ and vice versa (using the polarization identity). Moreover, the preceding result holds for any arbitrary domain $\Omega$. However if translation invariance is considered on some locally compact Abelian (LCA) group, this symmetry breaks as is demonstrated by the following discussion.

\begin{lem}
Consider a domain $\Omega \subseteq G$ where $G$ is some LCA group. Then every translation invariant positive definite kernel $K$ on $G$ yields a translation invariant Hilbertian metric $D$.
\end{lem}
\begin{proof}
By translation invariance of $K$ we have for all $\vecx,\vecy,\vec{t} \in G$, $K(\vecx+\vec{t},\vecy+\vec{t}) = K(\vecxy)$. Thus, for the corresponding Hilbertian metric $D^2(\vecxy) = K(\vecx,\vecx) + K(\vecy,\vecy) - 2K(\vecxy)$, we have $D^2(\vecx+\vec{t},\vecy+\vec{t}) = 2K(\veczero,\veczero) + K(\vecx+\vec{t},\vecy+\vec{t}) = 2K(\veczero,\veczero) + K(\vecxy) = D^2(\vecxy)$.
\end{proof}

However the converse does not hold true: take for example the familiar inner product on some finite dimensional Euclidean space $\R^d$ as the positive definite kernel. Its corresponding Hilbertian metric is the usual Euclidean metric $\norm{\cdot}_2$. However, whereas the metric (i.e. $\norm{\cdot}_2$) is translation invariant, the kernel (i.e. the inner product) is clearly not. In the sequel we shall explore this asymmetry by taking the example of the LCA group $\R$.

\section{Screw Functions and Positive Definite Kernels}
\Citet{helices} initiated an investigation that resulted in a complete characterization of translation invariant Hilbertian metrics on the real line. Recall that a metric $D$ defined on some domain $\Omega$ is said to be \emph{Hilbertian} if there exists a map $\xi : \Omega \rightarrow \H$ such that for all $\vecxy \in \Omega$, $D(\vecxy) = \norm{\xi(\vecx) - \xi(\vecy)}_\H$. \citet{helices} referred to the maps $\xi$ that generated these metrics as \emph{screw functions}\footnote{\Citet{helices} also completely characterized transforms of Euclidean metrics on finite dimensional spaces that are embeddable in $\H$, i.e. functions of the kind $F : \R^+ \rightarrow \R^+$ such that $D(\vecxy) = F(\norm{\vecx-\vecy})$ is a Hilbertian metric, thus completing a sequence of results that had previously characterized Hilbertian metric transforms for $\H$. It is important to note that only for $\R$ do these provide a complete characterization of all translation invariant Hilbertian metrics.}. Subsequently there were generalizations of this work to the more general class of \emph{homogeneous Hilbertian metrics} \citep{spirals}. However we shall refrain from expanding our scope beyond translation invariant Hilbertian metrics on the real line. Below we restate the first main theorem of \citet{helices}:

\begin{thm}[\citep{helices}, Theorem 1]
\label{vNS}
All and only those translation invariant metrics on $\R$ that can be expressed as the following Stieltjes integral are Hilbertian:
\[
D^2(x,y) = D^2(x-y) = \int_0^\infty \frac{\sin^2 t(x-y)}{t^2}d\gamma(t)
\]
where $\gamma(t)$ is a non-decreasing function on $\R^+$  such that $\displaystyle \int_1^\infty t^{-2}d\gamma(t)$ exists.
\end{thm}

Our main observation is that since translation invariant kernels all correspond to translation invariant Hilbertian metrics, the above result can be used to arrive at a characterization of translation invariant kernels on the real line as well. This result is famously known as Bochner's theorem \citep{faog}.  As a historical note we recall that although today Bochner's theorem is taken to be the result for general locally compact Abelian groups, it was originally proved for the group of integers $\Z$ by \citet{herglotz} and for the group of reals $\R$ by \citet{bochner}. The extension to locally compact Abelian groups is actually due to \citet{weil}. Bochner was the first to recognize the key role this result plays in harmonic analysis.

However, note that it is not possible to work things the other way round, i.e. use Bochner's theorem as a starting point to arrive at the result of \citet{helices} since not all translation invariant metrics on the real line correspond to translation invariant kernels as discussed above. This in some sense shows that, at least for the case of the real line, the result of \citet{helices} is a more general one than Bochner's theorem and hence the two are not equivalent.

However, Bochner's theorem does not follow trivially from the above discussion and the proof progression requires overcoming a few technical hurdles. Most importantly, one is require to qualify Theorem~\ref{vNS} with other preconditions since the relation between translation invariant kernels and Hilbertian metrics is not that of an exact bijection.

In the subsequent discussion we will, for sake of convenience, assume that the kernel $K$ is positive and normalized i.e. $K(x,y) \in [0,1]$ for all $x,y$. This restriction will allow us to present the main ingredients of the proof without being bogged down by stray factors. We shall later relax this condition to let $K$ assume arbitrary real values (positive as well as negative) in Appendix~\ref{app:general-proof}. Our aim shall be to establish Bochner's theorem in the following form:
\begin{thm}
\label{bochner-me}
Every real valued translation invariant positive definite kernel on $\R$ is the Fourier Stieltjes integral of some bounded positive Borel measure on $\R$.
\end{thm}

\section{Proof of Theorem~\ref{bochner-me}}
Suppose we are given a translation invariant kernel $K$ on the real line. We begin by constructing a distance measure $D$ on the real line as follows (recall the relation used in the proof of Corollary~\ref{psd-hilb}):
\[
D^2(x-y) = K(x,x) + K(y,y) - 2K(x,y),
\]
where $x,y \in \R$. Since $K$ is assumed to be translation invariant, $D$ is translation invariant as well. Applying Theorem~\ref{vNS} on $D$, we arrive at the following equalities
\begin{eqnarray*}
2K(0) - 2K(x,y) &=& \int_0^\infty \frac{\sin^2 t(x-y)}{t^2}d\gamma(t)\\
                   &=& \int_0^\infty \frac{1 - \cos 2t(x-y)}{2t^2}d\gamma(t)\\
                   &=& \int_0^\infty \frac{d\gamma(t)}{2t^2} - \int_0^\infty \frac{\cos 2t(x-y)}{2t^2}d\gamma(t).
\end{eqnarray*}
We realize that we might have committed blasphemy by this segregation of the two parts of the integral since Theorem~\ref{vNS} does not assure us that $\displaystyle\int_0^\infty t^{-2}d\gamma(t)$ exists. It only assures us that $\displaystyle \int_1^\infty t^{-2}d\gamma(t)$ exists. To remedy this, we use the following simple lemma:
\begin{lem}
\label{int-bound}
$\displaystyle \int_0^\infty \frac{d\gamma(t)}{t^2} \leq 4K(0)$
\end{lem}
\begin{proof}
\citet{helices} characterize bounded Hilbertian metrics by proving that for any such metric $D$, $\underset{t \rightarrow 0^+}\lim \gamma(t) = \gamma(0)$ i.e. $\gamma(t)$ does not have a discrete component at $0$ and that $\displaystyle \underset{t \rightarrow \infty}{\lim\sup}\ D^2(t) = \Theta\br{\int_{0}^\infty t^{-2}d\gamma(t)}$. More specifically, they prove
\[
\frac{1}{2}\int_{0}^\infty t^{-2}d\gamma(t) \leq \underset{t \rightarrow \infty}{\lim\sup}\ D^2(t) \leq \int_{0}^\infty t^{-2}d\gamma(t).
\]
In our case, we have $K(t) \in [0,1]$ which gives us
\[ 
D^2(t) = 2K(0) - 2K(t) \leq 2K(0).
\]
This proves that our Hilbertian metric is bounded which, upon applying the result of \citet{helices}, proves the claim.
\end{proof}
This justifies the segregation of the integral into two parts. Rearranging terms we get
\begin{eqnarray*}
K(x,y) &=& K(0) - \int_0^\infty \frac{d\gamma(t)}{4t^2} + \int_0^\infty \frac{\cos 2t(x-y)}{4t^2}d\gamma(t)\\
	   &=& C + \int_0^\infty \frac{\cos 2t(x-y)}{4t^2}d\gamma(t),
\end{eqnarray*}
where $\displaystyle C = K(0) - \int_0^\infty \frac{d\gamma(t)}{4t^2}$. Note that using Lemma~\ref{int-bound}, we have $C \geq 0$. Using a change of variables and a rescaling of the function $\gamma(\cdot)$ (which leaves all properties of $\gamma(\cdot)$ from Theorem~\ref{vNS} unharmed) gives us the following expression:
\[
K(x,y) = C + \int_0^\infty \frac{\cos t(x-y)}{t^2}d\gamma(t)
\]
%
for some constant $C \geq 0$.

Now we transform this Stieltjes integral into a Fourier Stieltjes integral to arrive at the canonical form of Bochner's theorem (for example as given in \citep{faog}). For any $0 \leq x < \infty$, define $\displaystyle\alpha(x) := \int_{0}^xt^{-2}d\gamma(t)$ which gives us
\[
K(x,y) = C + \int_0^\infty \cos t(x-y)d\alpha(t).
\]
Note that $\alpha\br{\cdot}$ is a non-decreasing right continuous function. Now define, for any interval $S = (a,b] \subset \R$, the measure $\mu$ as follows
\[
\mu(S) := \frac{\alpha(b) - \br{1 - 2\cdot\ind_{a < 0}}\alpha(a)}{2} + \frac{C}{2}\cdot\ind_{0 \in (a,b]}.
\]
An application of Carath\'eodory's extension theorem allows us to extend $\mu$ to a Borel measure over $\R$. Since $\alpha\br{\cdot}$ is a non decreasing function, $\displaystyle\int_0^\infty t^{-2}d\gamma(t) < \infty$ and $C > 0$, $\mu$ is a bounded positive measure and satisfies
\[
K(x,y) = \int_{-\infty}^\infty e^{it(x-y)}d\mu(t).
\]
This establishes Bochner's theorem on the real line as claimed.

\section{Extensions}
We note that the above characterization can also be extended to separable kernels on finite dimensional Euclidean spaces i.e. kernels which can be written as $K(\vecxy) = \prod\limits_{i=1}^dK_i(\vecx_i,\vecy_i)$ where $K_i(\vecx_i,\vecy_i)$ is a translation invariant kernel on the $i\th$ dimension. The resulting characterization is given as the following theorem:
\begin{thm}
Every separable real valued translation invariant positive definite kernel on $\R^d$ is the Fourier Stieltjes integral of some bounded positive Borel measure on $\R^d$.
\end{thm}
\begin{proof}
We have $K(\vecxy) = \prod\limits_{i=1}^dK_i(\vecx_i,\vecy_i)$ where every $K_i$ is a translation invariant kernel. Using Theorem~\ref{bochner-me}, there exist some positive bounded Borel measures $\mu_1,\ldots,\mu_n$ on $\R$ such that for all $i = 1, \ldots, d$, we have
\[
K_i(\vecx_i,\vecy_i) = \int_{-\infty}^\infty e^{it(\vecx_i-\vecy_i)}d\mu_i(t).
\]
Thus we have
\begin{eqnarray*}
K(\vecxy) &=& \prod_{i=1}^dK_i(\vecx_i,\vecy_i) = \prod_{i=1}^d\int_{-\infty}^\infty e^{it_i(\vecx_i-\vecy_i)}d\mu_i(t_i)\\
		  &=& \underbrace{\int_{-\infty}^\infty\ldots\int_{-\infty}^\infty}_d\br{\prod_{i=1}^de^{it_i(\vecx_i-\vecy_i)}}d\br{\prod_{i=1}^d\mu_i(t_i)} = \int_{\R^d}e^{i\ip{\vect}{\vecx-\vecy}}d\mu(\vect),
\end{eqnarray*}
where $\mu = \mu_1 \times \ldots \times \mu_d$ is a bounded positive product measure and we have used Fubini's theorem in the fourth step.
\end{proof}
This characterization applies to some of the most widely used translation kernels \citep{transinvembed} given below
\begin{enumerate}
	\item Gaussian kernel $\displaystyle K(\vecxy) = e^{-\frac{\norm{\vecx-\vecy}_2^2}{2}} = \prod_{i=1}^de^{-\frac{\br{\vecx_i-\vecy_i}^2}{2}}$.
	\item Laplacian kernel $\displaystyle K(\vecxy) = e^{-\norm{\vecx-\vecy}_1} = \prod_{i=1}^de^{-\abs{\vecx_i-\vecy_i}}$.
	\item Cauchy kernel $\displaystyle K(\vecxy) = \prod_{i=1}^d\frac{2}{1 + \br{\vecx_i-\vecy_i}^2}$.
\end{enumerate}
%

\appendix
\section{Generalizing the proof for arbitrary kernels}
\label{app:general-proof}
We now assume no normalizations on $K$, i.e. we allow for all $x,y$, $K(x,y) \in \R$. What this lack of normalization (more specifically, the ability of the kernel to take negative values) hinders is the ability to use $K(x,y) \geq 0$ to show that $\displaystyle C = K(0) -\int_0^\infty \frac{d\gamma(t)}{4t^2} \geq 0$ in the proof. Thus we are left, at the end of the day, with the following characterization:
\[
K(x,y) = \int_{-\infty}^\infty e^{it(x-y)}d\mu(t) + C,
\]
where $\mu$ is a bounded positive measure and $C$ is some (possibly negative) real constant. We first proceed by absorbing this constant into the measure $\mu$ by defining a new measure $\tilde\mu$. For any set $S = (a,b] \subset \R$, define
\[
\tilde\mu(S) := \mu(S) + C\cdot\ind_{0 \in (a,b]}.
\]
which gives us
\[
K(x,y) = \int_{-\infty}^\infty e^{it(x-y)}d\tilde\mu(t).
\]
Notice that this might have caused $\mu$ to acquire a discrete component at $0$ (notice that due to the property of $\gamma\br{\cdot}$ that $\underset{t \rightarrow 0^+}\lim \gamma(t) = \gamma(0)$, $\mu$ does not have a discrete component at $0$ to begin with). What we will show that unless this discrete component is non-negative, $K$ cannot be a positive definite kernel which will establish Bochner's theorem. 

This we shall show by proving that in case $C \neq 0$ (if $C = 0$ then we are done), then $\tilde\mu(0)$ is actually an eigenvalue of $K$. This will imply that for $K$ to be positive semi-definite, $\tilde\mu(0) \geq 0$ which shall complete our proof.

\begin{lem}
In case the modified distribution $\tilde\mu$ has a discrete component $\tilde\mu(0)$ at $0$ then $\tilde\mu(0)$ is an eigenvalue of $K$.
\end{lem}
\begin{proof}
For simplicity (and also without loss of generalization), assume that the domain of $K$ is the unit interval $[-1,1]$ so that the kernel is a function $K: [-1,1]\times[-1,1]\rightarrow\R$ (note that the kernel is still being allowed to take arbitrary real values). The purpose of proposing a compact domain is to allow us to work with constant functions $f(x) = c$ and still have $f(x) \in \ell^2$ with respect to the Lebesgue measure. Consider the integral operator $T_K$ on $\ell^2$ corresponding to the kernel $K$ defined as follows:
\[
T_K(f)(x) = \int_{- \infty}^\infty f(y)K(x,y)dy.
\]
We will show below that the constant function is an eigenfunction of this integral operator. Taking $f(x) = c$ we get
\begin{eqnarray*}
T_K(f)(x) &=& \int_{- \infty}^\infty f(y)K(x,y)dy = \int_{- \infty}^\infty c\int_{-\infty}^\infty e^{it(x-y)}d\tilde\mu(t)dy\\
		  &=& \int_{- \infty}^\infty ce^{itx}\br{\int_{-\infty}^\infty e^{-ity}dy}d\tilde\mu(t) = \int_{- \infty}^\infty ce^{itx}\delta(t)d\tilde\mu(t)\\
		  &=& c\tilde\mu(0) = \tilde\mu(0)f(x)
\end{eqnarray*}
where $\delta(x)$ is the Dirac delta function and we have used Fubini's theorem in the third step. This proves the claim.
\end{proof}

\bibliographystyle{plainnat}
\bibliography{ref}

\end{document}